\documentclass[10pt,reqno]{amsart}
\usepackage{latexsym,amsmath,amssymb,amscd}
\usepackage[all]{xy}
\usepackage{enumerate}

\def\today{\ifcase \month \or
   January \or February \or March \or April \or
   May \or June \or July \or August \or
   September \or October \or November \or December \fi
   \space\number\day , \number\year}
   
\makeatletter
  \newcommand\@dotsep{4.5}
  \def\@tocline#1#2#3#4#5#6#7{\relax
     \ifnum #1>\c@tocdepth 
     \else
     \par \addpenalty\@secpenalty\addvspace{#2}%
     \begingroup \hyphenpenalty\@M
     \@ifempty{#4}{%
     \@tempdima\csname r@tocindent\number#1\endcsname\relax
        }{%
         \@tempdima#4\relax
           }%
      \parindent\z@ \leftskip#3\relax \advance\leftskip\@tempdima\relax
      \rightskip\@pnumwidth plus1em \parfillskip-\@pnumwidth
       #5\leavevmode\hskip-\@tempdima #6\relax
       \leaders\hbox{$\m@th
       \mkern \@dotsep mu\hbox{.}\mkern \@dotsep mu$}\hfill
       \hbox to\@pnumwidth{\@tocpagenum{#7}}\par
       \nobreak
        \endgroup
         \fi}
\makeatother 

\begin{document}


\makeatletter
\@addtoreset{figure}{section}
\def\thefigure{\thesection.\@arabic\c@figure}
\def\fps@figure{h,t}
\@addtoreset{table}{bsection}

\def\thetable{\thesection.\@arabic\c@table}
\def\fps@table{h, t}
\@addtoreset{equation}{section}
\def\theequation{
\arabic{equation}}
\makeatother

\newcommand{\bfi}{\bfseries\itshape}

\newtheorem{theorem}{Theorem}
\newtheorem{acknowledgment}[theorem]{Acknowledgment}
\newtheorem{corollary}[theorem]{Corollary}
\newtheorem{definition}[theorem]{Definition}
\newtheorem{example}[theorem]{Example}
\newtheorem{lemma}[theorem]{Lemma}
\newtheorem{notation}[theorem]{Notation}
\newtheorem{problem}[theorem]{Problem}
\newtheorem{proposition}[theorem]{Proposition}
\newtheorem{question}[theorem]{Question}
\newtheorem{remark}[theorem]{Remark}
\newtheorem{setting}[theorem]{Setting}

\numberwithin{theorem}{section}
\numberwithin{equation}{section}

\renewcommand{\1}{{\bf 1}}
\newcommand{\Ad}{{\rm Ad}}
\newcommand{\Aut}{{\rm Aut}\,}
\newcommand{\ad}{{\rm ad}}
\newcommand{\botimes}{\bar{\otimes}}
\newcommand{\Ci}{{\mathcal C}^\infty}
\newcommand{\de}{{\rm d}}
\newcommand{\ee}{{\rm e}}
\newcommand{\End}{{\rm End}\,}
\newcommand{\id}{{\rm id}}
\newcommand{\ie}{{\rm i}}
\newcommand{\GL}{{\rm GL}}
\newcommand{\Gr}{{\rm Gr}}
\newcommand{\Hom}{{\rm Hom}\,}
\newcommand{\Ind}{{\rm Ind}}
\newcommand{\Ker}{{\rm Ker}\,}
\newcommand{\pr}{{\rm pr}}
\newcommand{\Ran}{{\rm Ran}\,}
\newcommand{\RRa}{{\rm RR}}
\newcommand{\rank}{{\rm rank}\,}
\renewcommand{\Re}{{\rm Re}\,}
\newcommand{\sa}{{\rm sa}}
\newcommand{\spa}{{\rm span}\,}
\newcommand{\tsr}{{\rm tsr}}
\newcommand{\Tr}{{\rm Tr}\,}

\newcommand{\CC}{{\mathbb C}}
\newcommand{\HH}{{\mathbb H}}
\newcommand{\RR}{{\mathbb R}}
\newcommand{\TT}{{\mathbb T}}

\newcommand{\Ac}{{\mathcal A}}
\newcommand{\Bc}{{\mathcal B}}
\newcommand{\Cc}{{\mathcal C}}
\newcommand{\Hc}{{\mathcal H}}
\newcommand{\Ic}{{\mathcal I}}
\newcommand{\Jc}{{\mathcal J}}
\newcommand{\Kc}{{\mathcal K}}
\newcommand{\Lc}{{\mathcal L}}
\renewcommand{\Mc}{{\mathcal M}}
\newcommand{\Nc}{{\mathcal N}}
\newcommand{\Oc}{{\mathcal O}}
\newcommand{\Pc}{{\mathcal P}}
\newcommand{\Vc}{{\mathcal V}}
\newcommand{\Xc}{{\mathcal X}}
\newcommand{\Yc}{{\mathcal Y}}
\newcommand{\Wc}{{\mathcal W}}

\renewcommand{\gg}{{\mathfrak g}}
\newcommand{\hg}{{\mathfrak h}}
\newcommand{\kg}{{\mathfrak k}}

\newcommand{\ZZ}{\mathbb Z}
\newcommand{\NN}{\mathbb N}

\makeatletter
\title[Real rank of $C^*$-algebras of solvable Lie groups]{On $C^*$-algebras of exponential solvable Lie groups 
and their real ranks}
\author{Ingrid Belti\c t\u a and Daniel Belti\c t\u a}
\address{Institute of Mathematics ``Simion Stoilow'' 
of the Romanian Academy, 
P.O. Box 1-764, Bucharest, Romania}
\email{ingrid.beltita@gmail.com, Ingrid.Beltita@imar.ro}
\email{beltita@gmail.com, Daniel.Beltita@imar.ro}
\thanks{This work was supported by a grant of the Romanian National Authority for Scientific Research and
Innovation, CNCS--UEFISCDI, project number PN-II-RU-TE-2014-4-0370}
\date{}
\makeatother

\begin{abstract} 
For any solvable Lie group whose exponential map $\exp_G\colon\gg\to G$ is bijective, 
we prove that the real rank of $C^*(G)$ is equal to $\dim(\gg/[\gg,\gg])$. 
We also indicate a proof of a similar formula for the stable rank of $C^*(G)$, as well as some estimates 
on the ideal generated by the projections in $C^*(G)$.
\end{abstract}

\maketitle


\section{Introduction}

The stable rank of $C^*$-algebras was introduced independently in \cite{CoLa83} and \cite{Ri83}, 
and subsequently, its self-adjoint version, called real rank, was introduced in \cite{BrPe91}, 
where it was proved among other things that the condition of positive real rank is equivalent to 
the fact that the linear span of projections is not dense in the $C^*$-algebra under consideration. 
Much information has been obtained on the real rank of $C^*$-algebras of locally compact groups, 
and yet some problems remained open so far. 
In the present note we answer some of these problems,  
based on the method of coadjoint orbits of exponential solvable Lie groups, 
particularly on the results of \cite{Cu92} 
(see also the monograph \cite{FuLu15} for background information on that method). 

First, there is  the problem of investigating  the set of projections in $C^*(G)$ for any exponential Lie group $G$. 
We recall that by definition an exponential Lie group is any Lie group whose 
exponential map $\exp_G\colon\gg\to G$ is bijective, and this condition implies that $G$ is a solvable Lie group. 
Examples of exponential Lie groups include the connected, simply connected, nilpotent Lie groups, 
for which it was established in \cite[Th. 4]{Su96} that their $C^*$-algebras contain no nonzero projections. 
The existence of projections has been left open so far for general exponential Lie groups, whose $C^*$-algebras 
are not always liminary. 
It is well known that if $G$ is the $(ax+b)$-group, then $C^*(G)$ contains non-trivial projections. 
We provide more precise information on this phenomenon in Theorem~\ref{th_pr} and Remark~\ref{rem_pr} below.

The next problem is motivated by the problem above and by 
related investigations in 
\cite{SuTa97}: Compute the real rank $\RRa(C^*(G))$ for any exponential Lie group~$G$.
See also \cite{ArKa12}, where an answer was provided in particular for nilpotent Lie groups.

The answer to this problem is given in Theorem~\ref{th_exp}, 
and the analogous result on stable ranks is given in Theorem~\ref{th_exp_tsr}, 
which also fills some gaps in the literature (see Remark~\ref{cor_exp_tsr}).

\section{Preliminaries}

\begin{definition}
\normalfont
Let $\Ac$ be any unital $C^*$-algebra, 
and for any integer $n\ge0$ denote by $\Lc_n(\Ac)$ the set of all $n$-tuples 
$(a_0,\dots,a_n)\in\Ac^n$ with $\Ac a_0+\cdots+\Ac a_n=\Ac$.  
We also denote $\Ac^{\sa}:=\{a\in\Ac\mid a=a^*\}$. 

The \emph{stable rank} of $\Ac$ is defined by 
$$\tsr(\Ac):=\min\{n\ge 1\mid \text{$\Lc_n(\Ac)$ is dense in $\Ac^n$}\}$$ 
with the usual convention $\min\emptyset=\infty$. 
The \emph{real rank} of $\Ac$ is similarly defined by 
$$\RRa(\Ac):=\min\{n\ge 0\mid \text{$\Lc_{n+1}(\Ac)\cap(\Ac^{\sa})^{n+1}$ is dense in $(\Ac^{\sa})^{n+1}$}\}.$$ 
For any non-unital $C^*$-algebra, its real rank and its stable rank are defined as 
the real rank, respectively the stable rank, of its unitization. 
\end{definition}

\begin{remark}\label{rem2.2}
\normalfont
The real rank of a $C^*$-algebra was introduced in \cite{BrPe91} as a noncommutative version 
of the covering dimension of a topological space. 
In this sense, it was established in \cite[Prop. 1.1]{BrPe91} that 
if $X$ is any compact Hausdorff space, then the real rank of the commutative $C^*$-algebra $\Cc(X)$ is equal to the covering dimension of $X$. 
 
On the other hand, by \cite[Th. 2.6]{BrPe91}, the real rank of some $C^*$-algebra $\Ac$ is zero 
if and only if every self-adjoint element in $\Ac$ is the limit of a sequence of self-adjoint elements with finite spectra. 
Therefore, if $\Ac$ contains no projections, then its real rank is $\ge1$, 
and this points out the connection between the two problems in the Introduction.
\end{remark}

The combination of \cite[Th. 4.11]{BBL14} 
($C^*$-algebras of nilpotent Lie groups are special solvable) 
with the following proposition gives a short proof of one of the main results of \cite{ArKa12} 
in the special case of Lie groups, namely that $\RRa(C^*(G)) =\dim (\gg/[\gg, \gg])$ for any connected, simply connected, nilpotent Lie group $G$. 
This equality will be extended in Theorem~\ref{th_exp} to all exponential Lie groups, using however a different approach, 
because the $C^*(G)$ in that case might not be liminary. 
We refer to \cite{BBL14} for the definition of special solvable $C^*$-algebras.

\begin{proposition}\label{spec_solv}
Let $\Ac$ be any liminary, solvable $C^*$-algebra with a special solving series 
$$\{0\}=\Jc_0\subseteq\Jc_1\subseteq\cdots\subseteq\Jc_n=\Ac$$
with $\Jc_j/\Jc_{j-1}\simeq\Cc_0(\Gamma_j,\Kc(\Hc_j))$ for $j=1,\dots,n$. 
Then $\RRa(\Ac)=\dim\Gamma_n$.  
\end{proposition}

\begin{proof}
The $C^*$-algebra $\Ac$ is liminary, hence it follows by \cite[Cor. 3.7]{Br07} 
that $\RRa(\Ac)=\max\{\RRa(\Jc_j/\Jc_{j-1})\mid j=1,\dots,n\}$. 
On the other hand, for $j=1,\dots,n-1$, we have $\dim\Hc_j=\infty$, hence 
$\RRa(\Cc_0(\Gamma_j,\Kc(\Hc_j))\le 1$ by \cite[Prop.~3.3]{BE91}. 
However $\dim\Hc_n=1$, hence $\Jc_n/\Jc_{n-1}\simeq\Cc_0(\Gamma_n)$ 
and then it follows by 
 \cite[Prop. 1.1]{BrPe91} that $\RRa(\Jc_n/\Jc_{n-1})=\dim \Gamma_n\ge 1$, 
since $\Gamma_n$ is homeomorphic to a vector space. 
Hence the assertion follows. 
\end{proof}

\section{Real rank for exponential Lie groups}

In this section we obtian one of our main results (Theorem~\ref{th_exp}), which gives a formula 
for computing $\RRa(C^*(G))$ in terms of the Lie algebra $\gg$, for ny exponential Lie group $G$, 
We must mention that the spectra finite-dimensionality hypotheses in Lemmas~\ref{RR1}--\ref{RR4}
are crucial for the conclusion of these lemma, and we need Lemma~\ref{dim} in order to check that 
these hypotheses are satisfied in the setting of the proof of Theorem~\ref{th_exp}.

\begin{lemma}\label{RR1}
Let $\Ac$ be any separable $C^*$-algebra with continuous trace, 
for which all its irreducible representations are infinite-dimensional. 
If moreover $\dim\widehat{\Ac}<\infty$, then $\RRa(\Ac)\le 1$. 
\end{lemma}

\begin{proof}
It follows by \cite[Cor. IV.1.7.22]{Bl06} that $\Ac$ is stable, 
hence there exists a $*$-isomorphism $\Ac\simeq \Ac\otimes\Kc$, 
where $\Kc$ is the $C^*$-algebra of compact operators on some separable infinite-dimensional Hilbert space. 
On the other hand, by \cite[Prop. 3.3]{BE91}, one has $\RRa(\Ac\otimes\Kc)\le 1$, 
and then $\RRa(\Ac)\le 1$ as well. 
\end{proof}

\begin{lemma}\label{RR4}
Let $\Ac$ be any $C^*$-algebra with an ideal $\Jc$ 
that is a continuous trace $C^*$-algebra with $\dim(\widehat{\Jc})< \infty$, and such that all irreducible representations of $\Jc$ are infinite dimensional.  
Then $\RRa(\Ac)=\max\{\RRa(\Jc),\RRa(\Ac/\Jc)\}$. 
\end{lemma}

\begin{proof}
This is a special case of \cite[Thm.~3.12(ii)]{Br07}.
\end{proof}

\begin{proposition}\label{RR5}
Let $\Ac$ be any $C^*$-algebra with a family of closed two-sided ideals 
$$\{0\}=\Jc_0\subseteq\Jc_1\subseteq\cdots\subseteq\Jc_n=\Ac$$
where for each for $j=2,\dots,n$,   $\Jc_j/\Jc_{j-1}$ has continuous trace, with its spectrum of finite covering dimension, and all its representations are infinite dimensional.   
Then 
$$\RRa(\Ac)=\max\{\RRa(\Jc_j/\Jc_{j-1})\mid j=1,\dots,n\}.$$  
\end{proposition}

\begin{proof}
We proceed by induction on $n$. 
The case $n=1$ is obvious. 
If we assume $n\ge 2$ and the assertion already proved for $n-1$, 
then, using the family of 
closed two-sided ideals 
$$\{0\}=\Jc_1/\Jc_1\subseteq\Jc_2/\Jc_1\subseteq\cdots\subseteq\Jc_n/\Jc_1=\Ac/\Jc_1$$
of $\Ac/\Jc_1$ for which $(\Jc_j/\Jc_1)/(\Jc_{j-1}/\Jc_1)\simeq \Jc_j/\Jc_{j-1}$ has Hausdorff spectrum,  for $j=2,\dots,n$, 
then the induction hypothesis implies 
$$\RRa(\Ac/\Jc_1)=\max\{\RRa(\Jc_j/\Jc_{j-1})\mid j=2,\dots,n\}.$$  
On the other hand, by Lemma~\ref{RR4} we have 
 $\RRa(\Ac)=\max\{\RRa(\Jc_1),\RRa(\Ac/\Jc_1)\}$, hence 
we directly obtain the assertion for $n$, and this completes the proof. 
\end{proof}

\begin{lemma}\label{dim}
Let $X$ be a metric space and $A$ a locally closed subset of $X$.
Then $\dim A\le \dim X$. 
\end{lemma}
\begin{proof}
The set $A$ is locally closed, hence there are sets $D\subseteq X$, $F\subseteq X$, with $D$ open and 
$F$ closed such that $A= D\cap F$. 

On the other hand $D$ is an $F_\sigma$ set in the metric space $X$, that is,  there is a countable family of closed subsets $F_n$, $n \ge 0$, such that $D= \cup_{n\ge 1} F_n$. 
Hence $A= \cup_{n\ge 1} (F_n \cap F)$, and from \cite[Prop.~3.1.5, Thm.~3.2.5, Thm.~3.2.7]{Pea75}
it follows that $\dim A\le \sup_{n\ge 1}\dim(F_n \cap F) \le \dim X$
\end{proof}

\begin{theorem}\label{th_exp}
For every exponential Lie group $G$ with its Lie algebra $\gg$, 
we have 
$$\RRa(C^*(G))=\dim(\gg/[\gg,\gg]).$$ 
\end{theorem}

\begin{proof}
Denoting $\Ac:=C^*(G)$, it follows by \cite[Cor. 3.2]{Cu92} that 
there exists a family of closed two-sided ideals 
$\{0\}=\Jc_0\subseteq\Jc_1\subseteq\cdots\subseteq\Jc_n=\Ac$
where $\Jc_j/\Jc_{j-1}$ is a separable continuous trace $C^*$-algebra 
whose irreducible representations are infinite-dimensional 
and for which $\Sigma_j:=\widehat{\Jc_j/\Jc_{j-1}}$ 
is homeomorphic to a semi-algebraic subset of $\gg^*$ for $j=1,\dots,n-1$. 
Moreover $\Ac/\Jc_{n-1}\simeq\Cc_0([\gg,\gg]^\perp)$, 
hence $\Sigma_n:=\widehat{\Ac/\Jc_{n-1}}$ is homeomorphic to the vector space 
$[\gg,\gg]^\perp\simeq (\gg/[\gg,\gg])^*\simeq \gg/[\gg,\gg]$. 
Then $\Sigma_j$ is a locally closed subset of $\gg^*$,  
hence it follows by Lemma~\ref{dim}
that $\dim\Sigma_j<\infty$.  
Thus by Lemma~\ref{RR1} we obtain $\RRa(\Jc_j/\Jc_{j-1})\le 1$ for $j=1,\dots,n-1$. 

On the other hand, denoting $r:=\dim(\gg/[\gg,\gg])$, 
it follows that the one-point compactification of $\Sigma_n$ is homeomorphic to the $r$-dimensional sphere $S^r$, 
hence the unitization of the $C^*$-algebra $\Ac/\Jc_{n-1}\simeq\Cc_0(\Sigma_n)$ 
is $*$-isomorphic to $\Cc(S^r)$.  
Using \cite[Prop. 1.1]{BrPe91}, we then obtain $\RRa(\Ac/\Jc_{n-1})=r=\dim(\gg/[\gg,\gg])$. 
Now, as an application of Proposition~\ref{RR5}, we obtain $\RRa(C^*(G))=\max\{r,1\}=r$, 
and this completes the proof. 
\end{proof}

\begin{corollary}\label{cor_exp}
Let $G$ be any connected Lie group with its Lie algebra $\gg$. 
If the universal covering group of $G$ is an exponential Lie group, 
then $\RRa(C^*(G))\le\dim(\gg/[\gg,\gg])$. 
\end{corollary}

\begin{proof}
Let $p\colon \widetilde{G}\to G$ be the universal covering map of $G$, 
so that we have a short exact sequence of Lie groups 
$$\1\to N\to\widetilde{G}\to G\to\1$$
where $N:=\Ker p$ is a discrete subgroup of the center of the exponential Lie group $\widetilde{G}$. 
Then all the groups involved in the above short exact sequence are amenable even as discrete groups, 
and we then obtain a short exact sequence of $C^*$-algebras 
$$0\to \Jc\to C^*(\widetilde{G})\to C^*(G)\to 0$$
for the ideal $\widehat{\Jc}=\{[\pi]\in\widehat{C^*(\widetilde{G})}\simeq\widehat{\widetilde{G}}\mid 
N\not\subset\Ker\pi\}$ of $C^*(\widetilde{G})$. 
We then obtain 
$$\RRa(C^*(G))=\RRa(C^*(\widetilde{G})/\Jc)\le \RRa(C^*(\widetilde{G}))=\dim(\gg/[\gg,\gg]).$$
The above inequality follows by \cite[Th. 1.4]{El95} (as in the proof of Lemma~\ref{RR4} above) 
and the final equality follows by Theorem~\ref{th_exp} applied for $\widetilde{G}$, 
taking also into account that the Lie algebra of $\widetilde{G}$ is isomorphic to~$\gg$. 
This concludes the proof. 
\end{proof}

\begin{remark}\label{rem_exp} 
\normalfont
The inequality in Corollary~\ref{cor_exp} can be strict if $G$ is not simply connected. 
For instance, if $G=\TT:=\RR/\ZZ$, then $C^*(G)=c_0(\NN)$ has real rank zero 
by \cite[Th. 2.6((i)$\Leftrightarrow$(ii))]{BrPe91}, 
so $\RRa(C^*(G))=0<1=\dim(\gg/[\gg,\gg])$.  
\end{remark}

\subsubsection*{Some remarks on abelianization}

Let $\Ac$ be any  $C^*$-algebra and denote by $\Jc(\Ac)$ its closed two-sided ideal generated by 
its subset of commutators $\spa\{[a,b]\mid a,b\in\Ac\}$. 
Then $\Ac/\Jc(\Ac)$ is a commutative $C^*$-algebra, 
hence there exists a locally compact space $\Gamma_{\Ac}$ 
(uniquely determined by $\Ac$ up to a homeomorphism) and  a $*$-isomorphism $\Ac/\Jc(\Ac)\simeq\Cc_0(\Gamma_{\Ac})$, 
and thus we obtain the short exact sequence 
\begin{equation}\label{ab0_eq1}
0\to\Jc(\Ac)\to\Ac\to\Cc_0(\Gamma_{\Ac})\to0.
\end{equation}
A natural question is to estimate the real rank of $\Ac$  in terms of the covering dimension $\dim(\Gamma_{\Ac}^\ast)$, 
where $\Gamma_\Ac^\ast$ denotes the one-point compactification of $\Gamma_\Ac$,  
We always have that  $\RRa(\Ac)\ge \dim(\Gamma_{\Ac}^\ast)$, by \eqref{ab0_eq1} and \cite[Th. 1.4]{El95}. 
The interesting equality $\RRa(\Ac)=\dim(\Gamma_{\Ac}^\ast)$ holds for several $C^*$-algebras: 
\begin{itemize}
\item If $\Ac=C^*(G)$ for an exponential Lie group $G$, 
then $\RRa(\Ac)=\dim(\Gamma_{\Ac})= \dim(\Gamma_{\Ac}^\ast)$ by Theorem~\ref{th_exp} 
and the short exact sequence in the proof of Theorem~\ref{th_pr}, which shows that $\Gamma_{\Ac}=[\gg,\gg]^\perp$. 
\item If $\Ac$ is any AF-algebra, then $\RRa(\Ac)=0=\dim(\Gamma_{\Ac}^\ast)$. 
In fact, this follows since it is well known that 
the quotients of AF-algebras are AF-algebras, and a commutative $C^*$-algebra is an AF-algebra 
if and only if its spectrum is totally disconnected, hence its covering dimension is equal to zero. 
\item If $\Ac$ is the $C^*$-algebra generated by the Toeplitz operators with continuous symbols on the unit circle $\TT$
(equivalently, $\Ac$ is the $C^*$-algebra generated by the unilateral shift operator), 
then one has $\Jc(\Ac)=\Kc(L^2(\TT))$ and $\Gamma_{\Ac}=\TT$, and it is known from \cite[Cor.~1.13(i)]{El95} that $\RRa(\Ac)=1=\dim(\TT)$. 
\end{itemize}
It is clear that the equality $\RRa(\Ac)=\dim(\Gamma_{\Ac}^\ast)$ fails to be true in general. 
For instance, if $\Ac$ is a simple $C^*$-algebra, then $\Jc(\Ac)=\Ac$, hence $\Gamma_{\Ac}=\emptyset$, 
and then $\dim(\Gamma_{\Ac}^\ast)=0$; on the other hand, examples of simple $C^*$-algebras are known, 
having positive real rank (see \cite[Th. 10]{Vi99}). 
Nevertheless, if $\Ac$ is a $C^*$-algebra of real rank zero and with  $\Jc(\Ac)\ne 0$, we have that 
$\dim (\Gamma_\Ac^\ast) =0$.

\section{On projections in the $C^*$-algebras of exponential Lie groups}

The main result of this section provide a kind of estimates on the size of the closed two-sided 
ideal generated by the projections in the $C^*$-algebra of an exponential Lie group $G$. 
That ideal is strictly smaller than $C^*(G)$, as already noted in Remark~\ref{rem2.2}.

\begin{notation}
\normalfont
For any $C^*$-algebra $\Ac$ we denote $\Gr(\Ac):=\{p\in\Ac\mid p=p^2=p^*\}$. 
\end{notation}

\begin{lemma}\label{pr1}
Let $\Ac$ be any $C^*$-algebra whose spectrum $\widehat{\Ac}$ is a Hausdorff space. 
If no connected component of $\widehat{\Ac}$ is compact, 
then $\Gr(\Ac)=\{0\}$. 
\end{lemma}

\begin{proof}
Let $p\in\Gr(\Ac)$ and $\Gamma$ be any connected component of $\widehat{\Ac}$. 
Since $\Gamma$ is not compact, it follows by \cite[3.3.7--9]{Dix64} that 
the function $\Gamma\to[0,\infty)$, $[\pi]\mapsto\Vert\pi(p)\Vert$, is continuous and 
$\lim\limits_{\Gamma\ni[\pi]\to\infty}\Vert\pi(p)\Vert =0$. 
But $\pi(p)\in\Bc(\Hc_\pi)$ is an orthogonal projection, hence $\Vert\pi(p)\Vert\in\{0,1\}$ 
for all $[\pi]\in\Gamma$. 
This implies that $\pi(p)=0$ for all $[\pi]\in\widehat{\Ac}$, 
and then by \cite[2.7.3]{Dix64} we obtain $p=0$, which completes the proof. 
\end{proof}

\begin{lemma}\label{pr2}
For any short exact sequence of $C^*$-algebras 
$$0\to\Jc\to\Ac\mathop{\to}\limits^Q\Ac/\Jc\to0$$ 
if $\Gr(\Ac/\Jc)=\{0\}$, then $\Gr(\Jc)=\Gr(\Ac)$. 
\end{lemma}

\begin{proof}
Since $\Jc\subseteq\Ac$, we have $\Gr(\Jc)\subseteq\Gr(\Ac)$. 
For the opposite inclusion,  
if $p\in\Gr(\Ac)$ then $Q(p)\in \Gr(\Ac/\Jc)=\{0\}$, 
hence $p\in\Jc$, and then $p\in\Gr(\Jc)$.  
\end{proof}

\begin{lemma}\label{pr3}
Let $\Ac$ be any $C^*$-algebra with a family of closed two-sided ideals 
$\Jc_1\subseteq\cdots\subseteq\Jc_n=\Ac$
where $\Jc_j/\Jc_{j-1}$ is a $C^*$-algebra 
whose spectrum $\widehat{\Jc_j/\Jc_{j-1}}$ is a Hausdorff space and such that 
no connected component of $\widehat{\Jc_j/\Jc_{j-1}}$ is compact for $j=2,\dots,n$. 
Then $\Gr(\Ac)=\Gr(\Jc_1)$.  
\end{lemma}

\begin{proof}
For $j=1,\dots,n$, one has the short exact sequence of $C^*$-algebras 
$$0\to\Jc_{j-1}\to\Jc_j\to\Jc_j/\Jc_{j-1}\to0.$$
Since $\widehat{\Jc_j/\Jc_{j-1}}$ is a Hausdorff space and 
no connected component of $\widehat{\Jc_j/\Jc_{j-1}}$ is compact, 
it follows by Lemma~\ref{pr3} that $\Gr(\Jc_j/\Jc_{j-1})=\{0\}$. 
Then, an application of Lemma~\ref{pr2} for the above exact sequence 
shows that $\Gr(\Jc_j)=\Gr(\Jc_{j-1})$ for $j=1,\dots,n$. 
But $\Jc_0=\{0\}$, hence $\Gr(\Jc_0)=\{0\}$. 
It then follows that 
$$\Gr(\Ac)=\Gr(\Jc_n)=\cdots=\Gr(\Jc_1),$$ 
and this completes the proof. 
\end{proof}

\begin{theorem}\label{th_pr}
Let $G$ be any exponential Lie group with its $C^*$-algebra $\Ac:=C^*(G)$. 
We denote by $\Jc_0\subset\Ac$ the intersection of kernels of all characters of $G$ 
extended to 1-dimensional $*$-representations of $\Ac$
Then $\Gr(\Ac)=\Gr(\Jc_0)$, and moreover $\Jc_0\subsetneqq\Ac$ if $\dim G>0$.  
\end{theorem}

\begin{proof}
It follows by the method of coadjoint orbits that we have a short exact sequence 
$$0\to\Jc_0\to\Ac\to\Cc_0([\gg,\gg]^\perp)\to0$$
where $[\gg,\gg]^\perp=\{\chi\in\gg^*\mid [\gg,\gg]\subseteq\Ker\chi\}$ 
is the space of characters of the Lie algebra~$\gg$. 
The vector space $[\gg,\gg]^\perp$ has no compact connected components, 
hence we may use Lemma~\ref{pr3} to obtain  $\Gr(\Ac)=\Gr(\Jc_0)$. 

Finally, since $\gg$ is a solvable Lie algebra, it follows that $[\gg,\gg]\subsetneqq\gg$ 
if $\dim\gg>0$, hence in this case $[\gg,\gg]^\perp\ne\{0\}$, and the above short exact sequence 
implies $\Jc_0\subsetneqq\Ac$, which concludes the proof. 
\end{proof}

The above Theorem~\ref{th_pr} shows that nontrivial projectors appear as soon as the group 
has open coadjoint orbits. 
The next proposition shows that there could be only a finite number of such orbits, and give a
a description of them.

\begin{proposition}\label{open}
Let $G$ be any connected Lie group with its Lie algebra $\gg$ and the duality pairing 
$\langle\cdot,\cdot\rangle\colon\gg^*\times\gg\to\RR$. 
For any basis $\{X_1,\dots,X_m\}$ in $\gg$ define the polynomial function
$$P\colon\gg^*\to\RR,\quad P(\xi):=\det(\langle\xi,[X_j,X_k]\rangle)_{1\le j,k\le m}.$$
Then the following assertions hold: 
\begin{enumerate}[(i)]
\item\label{open_item1} 
If $\xi\in\gg^*$, then the coadjoint orbit $\Oc_\xi:=\Ad^*_G(G)\xi$ 
is an open subset of $\gg^*$ if and only if $P(\xi)\ne0$. 
\item\label{open_item2} 
The set of open coadjoint orbits of $G$ is finite and their union is a Zarisky open subset of $\gg^*$ which may be empty. 
\end{enumerate}
\end{proposition}

\begin{proof}
For Assertion~\eqref{open_item1} 
denote $\gg(\xi):=\{X\in\gg\mid\langle\xi,[X,\cdot]\rangle=\{0\}\}$, the coadjoint isotropy subalgebra at $\xi$. 
Then the tangent space at $\xi\in\Oc_\xi$ can be computed as 
$T_\xi(\Oc_\xi)=\gg/\gg(\xi)$, hence $\Oc_\xi$ is an open subset of $\gg$ 
if and only if $\gg(\xi)=\{0\}$, and this is equivalent to the condition that the bilinear map 
$$\gg\times\gg\to\RR,\quad (X,Y)\mapsto B_\xi:=\langle\xi,[X,Y]\rangle $$
be nondegenerate. 
As $\{X_1,\dots,X_m\}$ is a basis in $\gg$, this condition is further equivalent to $P(\xi)\ne0$. 

For Assertion~\eqref{open_item2}, 
use Assertion~\eqref{open_item1} to see that 
the union of all open coadjoint orbits of $G$ is $P^{-1}(\RR\setminus\{0\})=\{\xi\in\gg^*\mid P(\xi)\ne0\}$, 
and this is a (maybe empty) Zarisky open subset of $\gg^*$ 
since $P\colon\gg^*\to\RR$ is a polynomial function. 

To see that the set of open coadjoint orbits of $G$ is finite, 
first note that every coadjoint orbit of $G$ is path connected since $G$ is path connected. 
Hence the open coadjoint orbits of $G$ can be equivalently described as the path connected 
components of the algebraic set $P^{-1}(\RR\setminus\{0\})$. 
Then we may use for instance \cite[Th. 4.1]{DK81}, 
which says in particular that the set of all path components of any algebraic variety over $\RR$ 
is finite, and this concludes the proof. 
\end{proof}

\begin{remark}
\normalfont
It follows by \cite[Th. 2.7 and Rem. 2.8]{Ko12} that the solvable Lie groups of type $HN$ 
that arise from Iwasawa decompositions of complex semisimple Lie groups 
(that is, Borel subgroups) 
may have at most one open coadjoint orbit, 
and such an open orbit exists if and only if $-1$ belongs to the corresponding Weyl group. 
\end{remark}

\begin{remark}[Theorem~\ref{th_pr} is sharp]\label{rem_pr} 
\normalfont
Let us resume the notation of Theorem~\ref{th_pr} and assume that 
the set $\Sigma_0$ of open coadjoint orbits of $G$ is nonempty. 
Then $\Sigma_0$ is a finite set and $\Jc_0$ contains a closed ideal $\Jc_{00}$ of $\Ac$ 
which is $*$-isomorphic to a direct sum of $\vert\Sigma_0\vert$ copies of the $C^*$-algebra $\Kc$ 
of compact operators on a separable infinite-dimensional complex Hilbert space, 
hence $\Gr(\Jc_{00})\ne\{0\}$. 

Moreover, the specific example of the $(ax+b)$-group shows that we may have $\Jc_{00}=\Jc_0$, 
hence in this case $\Gr(\Ac)=\Gr(\Jc_0)=\Gr(\Jc_{00})\ne\{0\}$. 
\end{remark}

\section{On the stable rank}

In this section we briefly indicate how the line of reasoning that leads to Theorem~\ref{th_exp} 
could be modified in order to compute the stable rank of $C^*(G)$ for any exponential Lie group $G$. 
This problem has been raised in \cite[Question 4.14]{Ri83}: 
If $G$ is a Lie group, how does one compute $\tsr(C^*(G))$ in terms of the
structure of $G$?
It was mentioned  already in \cite[Ex. 4.13]{Ri83} that if $G$ is the $(ax+b)$-group, then $\tsr(C^*(G))=2$.

\begin{remark}\label{tsr_prel}
\normalfont
For any compact space $X$ we have 
\begin{equation}\label{tsr_prel_eq1}
\tsr(\Cc(X))=1+[(\dim X)/2]
\end{equation} 
by \cite[Prop. 1.7]{Ri83}. 
For any $C^*$-algebra $\Ac$ and a separable infinite-dimensional complex Hilbert space~$\Hc$ we have 
\begin{equation}\label{tsr_prel_eq2}
\tsr(\Ac\otimes\Kc(\Hc)) 
\le 2
\end{equation}
by \cite[Th. 6.4]{Ri83}. 
Also, for any closed two-sided ideal $\Jc\subseteq\Ac$ we have 
\begin{equation}\label{tsr_prel_eq3}
\tsr(\Ac)\ge \max\{\tsr(\Jc),\tsr(\Ac/\Jc)\} 
\end{equation} 
by \cite[Th. 4.3--4.4]{Ri83}. 
  \end{remark}

\begin{remark}\label{tsr4}
\normalfont
In the setting of Lemma~\ref{RR4}, one has $\tsr(A)\le \max\{2, \tsr(A/J)\}$, 
by \cite[Th.~3.12(i)]{Br07}.  
\end{remark}

\begin{remark}\label{tsr5}
\normalfont Using Remark~\ref{tsr4} above, one can be prove by a similar method that, 
in the setting of Proposition~\ref{RR5}, one has 
$\tsr(A)\le \max\{2, \tsr(A/J_{n-1}) \}$.
 \end{remark}

\begin{theorem}\label{th_exp_tsr}
For every exponential Lie group $G$ with its Lie algebra $\gg$, 
if we denote $r:=\dim(\gg/[\gg,\gg])$, then  
$$\tsr(C^*(G))= 
\begin{cases}
1 &\text{ if and only if }G=\RR, \\
1+\max\{[r/2],1\} &\text{ otherwise}.
\end{cases}
$$ 
\end{theorem}

\begin{proof}
We resume the notation from the proof of Theorem~\ref{th_exp}.   
By Remark~\ref{tsr5} and Remark~\ref{tsr_prel} we obtain 
$$\tsr(\Ac/\Jc_{n-1}) \le \tsr(\Ac)\le \max\{2, \tsr(A/\Jc_{n-1})\}. $$
By \cite[Lemma 3.7]{SuTa97} we have that  $\tsr(\Ac)=1$ if and only if $G=\RR$. 
Assume that $\tsr(\Ac)\ge 2$. 
We have that  $\Ac/\Jc_{n-1}\simeq\Cc_0([\gg,\gg]^\perp)$, 
hence 
$\tsr(\Jc_n/\Jc_{n-1})=[r/2]+1$ by \eqref{tsr_prel_eq1}. 
Then, using also that $\max\{1+a,1+b\}=1+\max\{a,b\}$ for all $a,b\in\RR$, 
the conclusion follows directly. 
\end{proof}

\begin{remark}\label{cor_exp_tsr} 
\normalfont
 The formula provided by Theorem~\ref{th_exp_tsr} agrees with \cite[Th. 3.9]{SuTa97} 
in the case of exponential Lie groups. 
However, the proof of \cite[Th. 3.9]{SuTa97} seems to be incomplete 
because it is based on \cite[Lemma 3.2]{SuTa97}, 
whose proof requires 
the hypothesis 
on the finite-dimensionality of spectra of the continuous-trace algebras. 
A similar issue was pointed out at the top of \cite[page 100]{ArKa12}. 
\end{remark}

\end{document}